\documentclass[12pt,twoside]{amsart}
\usepackage{amssymb}
\usepackage{amscd}
\usepackage{xypic} 

\title{Vanishing theorems for toric polyhedra}
\author{Osamu Fujino} 
\date{2008/2/4}
\address{Graduate School of Mathematics, Nagoya University, 
Chikusa-ku Nagoya 464-8602 Japan}
\email{fujino@math.nagoya-u.ac.jp}
\newcommand{\Supp}[0]{{\operatorname{Supp}}}
\newcommand{\Hom}[0]{{\operatorname{Hom}}}
\newcommand{\Sing}[0]{{\operatorname{Sing}}}
\newcommand{\Spec}[0]{{\operatorname{Spec}}}
\newtheorem{thm}{Theorem}[section]
\newtheorem{lem}[thm]{Lemma}
\newtheorem{cor}[thm]{Corollary}
\newtheorem{prop}[thm]{Proposition}

\newtheorem*{claim}{Claim}

\newtheorem{thm-a}{Theorem}[subsection]
\newtheorem{lem-a}[thm-a]{Lemma}
\newtheorem{cor-a}[thm-a]{Corollary}
\newtheorem{prop-a}[thm-a]{Proposition}

\theoremstyle{definition}

\newtheorem{defn}[thm]{Definition}
\newtheorem{rem}[thm]{Remark}
\newtheorem*{ack}{Acknowledgments}      
\newtheorem*{notation}{Notation}         
\newtheorem{say}[thm]{}

\newtheorem{ex-a}[thm-a]{Example}
\newtheorem{que-a}[thm-a]{Question}
\newtheorem{defn-a}[thm-a]{Definition}
\newtheorem{rem-a}[thm-a]{Remark}
\newtheorem{say-a}[thm-a]{}

\begin{document}
\bibliographystyle{amsalpha+}

\maketitle

\begin{abstract} 
A toric polyhedron is a reduced closed 
subscheme of a toric variety 
that are partial unions of the orbits 
of the torus action. 
We prove vanishing theorems for toric polyhedra. 
We also give a proof of the $E_1$-degeneration 
of Hodge to de Rham type spectral sequence 
for toric polyhedra in any characteristic. 
Finally, we give a very powerful extension theorem for 
ample line bundles. 
\end{abstract}

\tableofcontents

\section{Introduction}\label{sec1} 
In this paper, we treat vanishing theorems for {\em{toric polyhedra}}. 
Section \ref{sec2} is a continuation of my paper \cite{fujino1}, 
where we gave a very simple, characteristic-free approach 
to vanishing theorems on toric varieties by using 
multiplication maps. 
Here, we give a generalization of Danilov's 
vanishing theorem on toric polyhedra. 

\begin{thm}[Vanishing Theorem]\label{00}
Let $Y=Y(\Phi)$ be a projective toric 
polyhedron defined over a field $k$ of arbitrary characteristic. Then 
$$
H^i(Y, \widetilde \Omega^a_Y\otimes L)=0 \ \ \text{for} \ \ i\ne 0
$$ 
holds for every ample line bundle $L$ on $Y$. 
\end{thm}

Note that a toric polyhedron is a reduced closed subscheme 
of a toric variety that are partial 
unions of the orbits of the torus action. 
Once we understand Ishida's de Rham complexes 
on toric polyhedra, then we can easily see that 
the arguments in \cite{fujino1} works for toric 
polyhedra with only small modifications. 
Moreover, we give a proof of the $E_1$-degeneration 
of Hodge to de Rham type spectral sequence 
for toric polyhedra. 

\begin{thm}[$E_1$-degeneration]\label{01} 
Let $Y=Y(\Phi)$ be a complete 
toric polyhedron defined over a field $k$ of any characteristic. 
Then the spectral sequence 
$$
E^{a,b}_1=H^b(Y, \widetilde \Omega^a_Y)
\Rightarrow \mathbb H^{a+b} 
(Y, \widetilde \Omega ^{\bullet}_Y)
$$
degenerates at the $E_1$-term. 
\end{thm}

It seems to be new when 
the characteristic of the base field is positive.  
So, Section \ref{sec2} supplements \cite{btlm}, 
\cite{danilov}, and \cite{ishida}. 
In Section \ref{sec3}, we will give the following two 
results supplementary 
to \cite{fujino1}.  

\begin{thm}[{cf.~\cite[Theorem 1.1]{fujino1}}]\label{b-I}
Let $X$ be a toric variety defined over 
a field $k$ of any characteristic and let $A$ and $B$ be reduced torus 
invariant Weil divisors on $X$ without common irreducible 
components. 
Let $L$ be a line bundle on $X$. If 
$H^i(X, \widetilde {\Omega}^a_{X}(\log(A+B))(-A)\otimes L^
{\otimes l})=0$ for 
some positive integer $l$, 
then $H^i(X, \widetilde {\Omega}^a_{X}(\log (A+B))(-A)\otimes L)=0$. 
\end{thm}

It is a slight generalization of \cite[Theorem 1.1]{fujino1}. 

\begin{thm}\label{b-II} 
Let $X$ be a complete 
toric variety defined over a field $k$ of any 
characteristic and let $A$ and $B$ be reduced torus 
invariant Weil divisors on $X$ without common irreducible 
components. 
Then the spectral sequence 
$$
E^{a,b}_1=H^b(X, \widetilde \Omega^a_{X}(\log (A+B))(-A))
\Rightarrow \mathbb H^{a+b}(X, \widetilde \Omega^{\bullet}
_X(\log (A+B))(-A)) 
$$ 
degenerates at the $E_1$-term. 
\end{thm}

One of the main results of this paper is the 
next theorem, which is a complete generalization of 
\cite[Theorem 5.1]{mustata}. 
For the precise statement, see Theorem \ref{531} below. 
We will give a proof of Theorem \ref{ext-thm} 
as an application of our new vanishing arguments 
in Section \ref{sec-kol}. 
The technique in Section \ref{sec-kol} 
is very powerful and produces Koll\'ar type 
vanishing theorem in the toric category, 
which is missing in \cite{fujino1}. 

\begin{thm}[Extension Theorem]\label{ext-thm} 
Let $X$ be a projective 
toric variety defined over a field $k$ of any characteristic and 
let $L$ be an ample line bundle on $X$. 
Let $Y$ be a toric polyhedron on $X$ and 
let $\mathcal I_Y$ be the defining ideal sheaf 
of $Y$ on $X$. 
Then $H^i(X, \mathcal I_Y\otimes L)=0$ 
for any $i>0$. 
In particular, the restriction map $H^0(X, L)\to 
H^0(Y, L)$ is surjective. 
\end{thm}

We state a special case of the vanishing theorems in Section \ref{sec-kol} 
for the reader's convenience. 

\begin{thm}[cf.~Theorem \ref{sugo}] 
Let $f:Z\to X$ be a toric morphism between 
projective toric varieties and let $A$ and $B$ be reduced 
torus invariant Weil divisors on $Z$ without common irreducible 
components. Let $L$ be 
an ample line bundle on $X$. 
Then $H^i(X, L\otimes R^jf_*\widetilde {\Omega}^a_Z
(\log (A+B))(-A))=0$ for any $i>0$, $j\geq 0$, and 
$a\geq 0$. 
\end{thm}

In the final section:~Section \ref{sec4}, we treat toric polyhdera 
as {\em{quasi-log varieties}} and explain the background 
and motivation of this work. 
 
\begin{ack}
I was partially supported by the Grant-in-Aid for Young Scientists 
(A) $\sharp$17684001 from JSPS. I was 
also supported by the Inamori Foundation. 
I am grateful to Professor Shigefumi Mori for his questions, 
comments, and warm 
encouragement. 
I thank Hiroshi Sato and the referee for their comments. 
I also thank Takeshi Abe for answering my question. 
\end{ack}

\begin{notation}
Let $N$ be a free $\mathbb Z$-module of rank $n\geq 0$ and 
let $M$ be its dual $\mathbb Z$-module. 
The natural pairing $\langle \ , \ \rangle: N\times M\to 
\mathbb Z$ is extended to the bilinear 
form $\langle \ , \ \rangle: N_{\mathbb R}\times 
M_{\mathbb R}\to \mathbb R$, where 
$N_{\mathbb R}=N\otimes _{\mathbb Z}\mathbb R$ and 
$M_{\mathbb R}=M\otimes _{\mathbb Z}\mathbb R$. 
A non-empty subset $\sigma$ 
of $N_{\mathbb R}$ is said to 
be a {\em{cone}} if there 
exists a finite subset $\{n_1, \cdots, n_s\}$ of $N$ such that 
$\sigma =\mathbb R_{\geq 0}n_1+\cdots 
+\mathbb R_{\geq 0} n_s$, where 
$\mathbb R_{\geq 0}=\{ r\in \mathbb R; r \geq 0\}$, 
and that $\sigma \cap (-\sigma)=\{0\}$, 
where $-\sigma =\{-a; a\in \sigma\}$. 
A subset $\rho$ of a cone $\sigma$ is said to be 
a {\em{face}} of $\sigma$ and we denote $\rho\prec 
\sigma$ if 
there exists an element $m$ of $M_{\mathbb R}$ such that 
$\langle a, m \rangle\geq 0$ for every $a\in \sigma$ and 
$\rho =\{a \in \sigma; \langle a, m \rangle =0\}$. 
A set $\Delta$ of cones of $N_\mathbb R$ is said to 
be a {\em{fan}} if (1) $\sigma\in \Delta$ and 
$\rho\prec \sigma$ imply $\rho \in \Delta$, and 
(2) $\sigma, \tau \in \Delta$ and $\rho =\sigma 
\cap \tau$ imply $\rho \prec \sigma$ and $\rho \prec \tau$. 
We do not assume that $\Delta$ is {\em{finite}}, 
that is, $\Delta$ does not always consist 
of a finite number of 
cones. 
For a cone $\sigma$ of $N_{\mathbb R}$, 
$\sigma ^{\vee} =\{ x \in M_{\mathbb R}; 
\langle a, x\rangle \geq 0 \ \text{for every} \ a \in \sigma\}$ and 
$\sigma ^{\perp} =\{ x \in M_{\mathbb R}; 
\langle a, x\rangle = 0 \ \text{for every} \ a \in \sigma\}$. 
Let $X=X(\Delta)$ be the {\em{toric variety}} 
associated to 
a fan $\Delta$. 
Note that $X$ is just {\em{locally}} 
of finite type over $k$ in our notation, 
where $k$ is the base field of $X(\Delta)$. 
Each cone $\sigma$ of $\Delta$ uniquely defines 
an $(n-\dim \sigma)$-dimensional 
torus $T_{N(\sigma)}=\Spec k [M\cap 
\sigma ^{\perp}]$ on $X(\Delta)$. 
The closure of $T_{N(\sigma)}$ in $X(\Delta)$ 
is denoted by $V(\sigma)$. 
\end{notation}

\section{Vanishing theorem and $E_1$-degeneration}\label{sec2}  

We will work over a fixed field $k$ of any characteristic throughout this 
section. 
 
\subsection{Toric polyhedra} 
Let us recall the definition of {\em{toric polyhedra}}. 
See \cite[Definition 3.5]{ishida}. 

\begin{defn-a}\label{10} 
For a subset $\Phi$ of a fan $\Delta$, we say that 
$\Phi$ is {\em{star closed}} 
if $\sigma \in \Phi$, $\tau \in \Delta$ and 
$\sigma \prec \tau$ imply $\tau \in \Phi$. 
\end{defn-a}

\begin{defn-a}[Toric polyhedron]\label{11} 
For a star closed subset $\Phi$ of a fan $\Delta$, 
we denote by $Y=Y(\Phi)$ the reduced 
subscheme $\bigcup _{\sigma\in \Phi} V(\sigma)$ of 
$X=X(\Delta)$, and we call it 
the {\em{toric polyhedron associated to 
$\Phi$}}. 
\end{defn-a}

\begin{ex-a}\label{23} 
Let $X=\mathbb P^2$ and let $T\subset \mathbb P^2$ be 
the big torus. We put $Y=\mathbb P^2\setminus T$. 
Then $Y$ is a toric polyhedron, which is a circle of 
three projective lines. 
\end{ex-a}

The above example is a special case of the following 
one. 

\begin{ex-a}\label{231} 
Let $X=X(\Delta)$ be an $n$-dimensional toric variety. 
We put $\Phi_m=\{ \sigma \in \Delta; 
\dim \sigma \geq m\}$ for $0\leq m\leq n$. 
Then $\Phi_m$ is a star closed subset of $\Delta$ and 
the toric polyhedron $Y_m=Y(\Phi_m)$ is pure 
$(n-m)$-dimensional. 
\end{ex-a}

\begin{ex-a}
We consider $X=\mathbb A^3_k=\Spec 
k [x_1, x_2, x_3]$. 
Then the subvariety $Y=(x_1=x_2=0)\cup 
(x_3=0)\simeq \mathbb A^1_k\cup 
\mathbb A^2_k$ of $X$ is a toric polyhedron, which is not pure dimensional. 
\end{ex-a}

\begin{rem-a}\label{24}
Let $Y$ be a toric polyhedron. 
We do not know how to describe line bundles on $Y$ by 
combinatorial 
data. Note that a line bundle $L$ on $Y$ can not necessarily be 
extended 
to a line bundle $\mathcal L$ on $X$.  
\end{rem-a}

In \cite{ishida}, Ishida defined the de Rham complex 
$\widetilde \Omega^{\bullet}_Y$ of a toric polyhedron $Y$. 
When $Y$ is a toric variety, Ishida's de Rham complex is 
nothing but Danilov's de Rham complex (see \cite[Chapter I.~\S4]
{danilov}). 
For the details, see \cite{ishida}. 
Here, we quickly review $\widetilde \Omega^{\bullet}_Y$ when 
$X$ is affine. 

\begin{say-a}[Ishida's de Rham complex]\label{de-r} 
We put $\Delta=\{\pi, \ {\text{its faces}}\}$, 
where $\pi$ is a cone in $N_{\mathbb R}$. 
Then 
$X=X(\Delta)$ is an affine toric 
variety $\Spec k[M\cap \pi^{\vee}]$. Let $\Phi$ 
be a star closed subset of $\Delta$ and 
let $Y$ be the toric polyhedron associated to $\Phi$. 
In this case, $\widetilde \Omega^a_Y$ is 
an $\mathcal O_Y=k[M\cap \pi^{\vee}]/k[M\cap 
(\pi^{\vee}\setminus (\cup _{\sigma \in \Phi}
\sigma ^{\perp}))]$-module 
generated by $x^m\otimes m_{\alpha_1}\wedge 
\cdots \wedge m_{\alpha_a}$, where 
$m\in M\cap (\pi^{\vee}\cap 
(\cup _{\sigma \in \Phi}\sigma ^{\perp}))$ and 
$m_{\alpha_1}, \cdots, m_{\alpha_a}\in M[\rho (m)]$, 
for any $a\geq 0$. 
Note that $\rho (m)=\pi\cap m^{\perp}$ is 
a face of $\pi$ when 
$m\in M\cap \pi^{\vee}$, and 
that $M[\rho(m)]=M\cap \rho (m)^{\perp}\subset M$. 
\end{say-a}

\subsection{Multiplication maps}\label{sub22} 
In this subsection, let us quickly review 
the multiplication maps in \cite[Section 2]{fujino1}. 

\begin{say-a}[Multiplication maps]\label{o21}
For a fan $\Delta$ in $N_{\mathbb R}$, we 
have the associated toric variety $X=X(\Delta)$. 
We put $N'=\frac{1}{l}N$ and $M'=\Hom _{\mathbb Z}(N', \mathbb Z)$ for 
any positive integer $l$. 
We note that $M'=lM$. 
Since $N_{\mathbb R}=N'_{\mathbb R}$, $\Delta$ is also a fan in 
$N'_{\mathbb R}$. 
We write $\Delta'$ to express the fan $\Delta$ in $N'_{\mathbb R}$. 
Let $X'=X(\Delta')$ be the associated toric variety. 
We note that $X\simeq X'$ as toric varieties. 
We consider 
the natural inclusion $\varphi:N\to N'$. 
Then $\varphi$ induces a finite surjective toric morphism $F:X\to X'$. 
We call it the {\em{$l$-times multiplication map}} of $X$. 
\end{say-a} 

\begin{rem}The $l$-times multiplication map $f:X\to X'$ should 
be called the {\em{$l$-times multiplication map}} of $X$. However, we follow 
\cite{fujino1} in this paper. 
\end{rem}

\begin{say-a}[Convention]\label{o22}
Let $\mathcal A$ be an object on $X$. 
Then we write $\mathcal A'$ 
to indicate the corresponding 
object on $X'$. 
Let $\Phi$ be a star closed subset of 
$\Delta$ and 
let $Y$ be the toric polyhedron associated to 
$\Phi$. 
Then $F:X\to X'$ induces a finite surjective morphism 
$F:Y\to Y'$. 
\end{say-a}

\begin{say-a}[Split injections on the big torus]\label{o23} 
By fixing a base of $M$, we have 
$k[M]\simeq k[x_1, x^{-1}_1, \cdots, x_n, x^{-1}_n]$. 
We can write $x^m=x^{m_1}_1 x^{m_2}_2
\cdots x^{m_n}_n$ for 
$m=(m_1, \cdots, m_n)\in \mathbb Z^n=M$. 
Let $T$ be the big torus of $X$. 
Then we have the isomorphism of $\mathcal O_T=k[M]$-modules 
$k[M]\otimes _{\mathbb Z}\wedge ^a M
\to \Omega^a_T$ for any $a\geq 0$ 
induced by 
$$x^{m}\otimes m_{\alpha_1}\wedge 
\cdots \wedge m_{\alpha_a} 
\mapsto x^{m}
\frac{dx^{m_{\alpha_1}}}{x^{m_{\alpha_1}}}\wedge 
\cdots \wedge 
\frac{dx^{m_{\alpha_a}}}{x^{m_{\alpha_a}}},$$  
where $m, m_{\alpha_1}, \cdots, m_{\alpha_a}
\in \mathbb Z^n=M$. 
Therefore, $F_*\Omega^a_T$ 
corresponds to a $k[M']$-module $k[M]\otimes _{\mathbb Z}\wedge ^aM$. 
We consider the $k[M']$-module 
homomorphisms $k[M']\otimes _{\mathbb Z}\wedge ^aM'
\to k[M]\otimes _{\mathbb Z}\wedge ^aM$ given by 
$x^{m_{\beta}}\otimes m_{\alpha_1}\wedge \cdots \wedge 
m_{\alpha_a} \mapsto x^{lm_{\beta}}\otimes 
m_{\alpha_1}\wedge \cdots \wedge 
m_{\alpha_a}$, 
and $k[M]\otimes _{\mathbb Z}\wedge ^a M\to 
k[M']\otimes _{\mathbb Z}\wedge^aM'$ induced by 
$x^{m_{\gamma}}\otimes m_{\alpha_1}\wedge \cdots \wedge 
m_{\alpha_a}\mapsto x^{m_{\beta}}\otimes 
m_{\alpha_1}\wedge \cdots \wedge 
m_{\alpha_a}$ if $m_{\gamma}=lm_{\beta}$ 
and $x^{m_{\gamma}}\otimes m_{\alpha_1}\wedge \cdots \wedge 
m_{\alpha_a}\mapsto 0$ otherwise. 
Thus, these $k[M']$-module homomorphisms give 
split injections 
$\Omega^a_{T'}\to F_*\Omega^a_T$ for any $a\geq 0$. 
\end{say-a}

\subsection{Proof of the vanishing theorem and $E_1$-degeneration} 
Let us start the proof of the vanishing theorem and $E_1$-degeneration. 
The next proposition plays a key role in the proof. 

\begin{prop-a}\label{135} 
Let $X$ be a toric variety and let $Y\subset X$ be a toric polyhedron. 
Let $F:X\to X'$ be the $l$-times multiplication map 
and let $F:Y\to Y'$ be the induced map. 
Then there exists a split injection 
$\widetilde \Omega^a_{Y'}\to F_*\widetilde \Omega^a_Y$ for 
any $a\geq 0$. 
\end{prop-a}
\begin{proof}
We write $X=X(\Delta)$ and $Y=Y(\Phi)$. 
Then $Y(\Phi)$ has the open covering 
$\{ Y(\Phi)\cap U(\pi); \pi \in \Phi\}$, where 
$U(\pi)=\Spec k [M\cap \pi^{\vee}]$. 
We put $Z=Z(\Psi)=Y(\Phi)\cap 
U(\pi)$. 
Then, by the description of $\widetilde \Omega^a_Z$ in 
\cite[Section 2]{ishida} or \ref{de-r}, 
we have natural embeddings $\mathcal O_Z\subset 
k[M]$ and $\widetilde \Omega^a_Z\subset k[M]\otimes 
_{\mathbb Z}\wedge ^aM$ for any $a>0$ as $k$-vector spaces. 
Note that $\mathcal O_Z$ is spanned by 
$\{x^m; m\in M\cap (\pi^{\vee}\cap (\cup _{\sigma \in 
\Psi}\sigma^{\perp}))\}$ 
as a $k$-vector space. 
In \ref{o23}, 
we constructed split injections 
$k[M']\otimes _{\mathbb Z}\wedge ^aM'\to 
k[M]\otimes _{\mathbb Z}\wedge ^aM$ for 
any $a\geq 0$. 
This split injections induce 
split injections 
$$
\begin{array}{ccc} 
\widetilde \Omega^a_{Z'}&\to&F_*\widetilde \Omega^a_Z\\ 
\cap & &\cap \\
k[M']\otimes _{\mathbb Z}\wedge ^a M' &\to 
& k[M]\otimes _{\mathbb Z}\wedge ^a M
\end{array}
$$ 
for all $a\geq 0$ as $k$-vector spaces. 
However, it is not difficult to see 
that $\widetilde \Omega^a_{Z'}\to F_*
\widetilde \Omega^a_Z$ 
and its split $F_*\widetilde \Omega^a_Z
\to \widetilde \Omega^a_{Z'}$ 
are $\mathcal O_{Z'}$-homomorphisms for 
any $a\geq 0$. 
The above constructed split injections 
for $Y(\Phi)\cap U(\pi)$ can be patched together. 
Thus, we 
obtain split injections $\widetilde 
\Omega^a_{Y'}\to F_*\widetilde \Omega^a_Y$ for 
any $a\geq 0$. 
\end{proof}

The following theorem is one of the main theorems 
of this paper. 
It is a generalization of Danilov's vanishing theorem 
for toric varieties (see \cite[7.5.2. Theorem]{danilov}). 

\begin{thm-a}[cf.~Theorem \ref{00}]\label{12}
Let $Y=Y(\Phi)$ be a projective toric 
polyhedron defined over a field $k$ of any characteristic. Then 
$$
H^i(Y, \widetilde \Omega^a_Y\otimes L)=0 \ \ \text{for} \ \ i\ne 0
$$ 
holds for every ample line bundle $L$ on $Y$. 
\end{thm-a}
\begin{proof}
We assume that $l=p>0$, where 
$p$ is the characteristic of $k$. 
In this case, $F^*L'\simeq L^{\otimes p}$. Thus, 
we obtain 
\begin{align*}
H^i(Y, \widetilde \Omega^a_Y\otimes L)&\simeq 
H^i(Y', \widetilde \Omega^a_{Y'}\otimes L')
\\ 
&\subset 
H^i(Y', F_*\widetilde \Omega^a_Y\otimes L')\\&\simeq 
H^i(Y, \widetilde \Omega^a_Y\otimes L^{\otimes p}), 
\end{align*} 
where we used the split injection in Proposition \ref{135} 
and the projection formula. 
By iterating 
the above arguments, 
we obtain 
$H^i(Y, \widetilde \Omega^a_Y\otimes L)\subset 
H^i(Y, \widetilde \Omega^a_Y\otimes L^{\otimes p^r})$ 
for any positive integer $r$. 
By Serre's vanishing theorem, we obtain 
$H^i(Y, \widetilde \Omega^a_Y\otimes L)=
H^i(Y, \widetilde \Omega^a_Y\otimes L^{\otimes p^r})=0$ for 
$i>0$. 
When the characteristic of $k$ is zero, 
we can assume that everything is defined over $R$, 
where 
$R (\supset \mathbb Z)$ is a finitely generated ring. 
By the above result, the vanishing theorem holds 
over $R/P$, 
where $P$ is any general maximal 
ideal of $R$, since $R/P$ is a finite field and 
the ampleness is an open condition. 
Therefore, 
we have the desired vanishing theorem over the 
generic point of $\Spec R$. 
Of course, it holds over $k$. 
\end{proof}

If $Y$ is a toric variety, then Theorem \ref{12} is nothing 
but Danilov's vanishing theorem. 
For the other vanishing theorems on 
toric varieties, see \cite{fujino1} and 
the results in Sections \ref{sec3} and \ref{sec-kol}. 
The next corollary 
is a special case of Theorem \ref{12}. 

\begin{cor-a}\label{13}
Let $Y=Y(\Phi)$ be a projective toric 
polyhedron and $L$ an ample line bundle on 
$Y$. Then we obtain 
$H^i(Y, L)=0$ for any $i>0$. 
\end{cor-a}
\begin{proof}
It is sufficient to remember that $\widetilde \Omega^0_Y
\simeq \mathcal O_Y$. 
\end{proof}

\begin{rem-a}
Let $X$ be a projective toric variety. Then, it is obvious that 
$H^i(X, \mathcal O_X)=0$ for $i>0$. However, 
$H^i(Y, \mathcal O_Y)$ is not necessarily zero 
for some $i>0$ when $Y$ is a projective toric 
polyhedron. See Example \ref{23}. 
More explicitly, let $X$ be an $n$-dimensional 
non-singular complete toric variety. We put 
$Y=X\setminus T$, where 
$T$ is the big torus. 
Then $H^{n-1}(Y, \mathcal O_Y)$ is dual to 
$H^0(Y, \mathcal O_Y)$ since 
$K_Y\sim 0$. Therefore, 
$H^{n-1}(Y, \mathcal O_Y)\ne \{0\}$. 
\end{rem-a}

The following theorem is a supplement 
to Theorem \ref{12}. 
\begin{thm-a}
Let $Y$ be a toric polyhedron on a toric 
variety $X$. 
Let $L$ be a line bundle on $Y$. Assume that 
$L=\mathcal L|_Y$ for some line bundle $\mathcal L$ on $X$. 
If $H^i(Y, \widetilde \Omega^a_Y\otimes L^{\otimes l})=0$ for 
some positive integer $l$, 
then $H^i(Y, \widetilde \Omega^a_Y\otimes 
L)=0$.  
\end{thm-a}
\begin{proof}
Let $F:X\to X'$ be the $l$-times 
multiplication map. Then $F^*\mathcal L'\simeq 
\mathcal L^{\otimes l}$. Therefore, 
$F^*L'\simeq L^{\otimes l}$. By the 
same argument as in the proof of 
Theorem \ref{12}, we obtain the desired statement.  
\end{proof}

By the construction of the split injections 
in Proposition \ref{135} and 
the definition of the exterior derivative, we have the following 
proposition. 

\begin{prop-a}\label{14} 
We assume that $l=p>0$, where 
$p$ is the characteristic of $k$. 
Then there exist morphisms of complexes 
$$
\phi:\bigoplus_{a\geq 0} \widetilde \Omega ^a_{Y'}[-a]\to 
F_*\widetilde \Omega ^{\bullet}_Y 
$$ 
and 
$$
\psi: 
F_*\widetilde \Omega ^{\bullet}_Y\to 
\bigoplus_{a\geq 0} \widetilde \Omega ^a_{Y'}[-a] 
$$ 
such that $\psi\circ \phi$ is 
a quasi-isomorphism. 
Note that 
the complex 
$\bigoplus_{a\geq 0} \widetilde \Omega ^a_{Y'}[-a]$ 
has zero differentials. 
\end{prop-a}
\begin{proof}
We consider the following diagram. 
$$
\begin{CD}
\widetilde\Omega^{a}_{Y'}@>{\phi_a}>>F_*\widetilde \Omega^{a}_Y@>{\psi_a}>> 
\widetilde \Omega^a_{Y'}\\ 
@V{0}VV @VV{d}V @VV{0}V\\
\widetilde \Omega^{a+1}_{Y'}@>>{\phi_{a+1}}>
F_*\widetilde \Omega^{a+1}_Y@>>{\psi_{a+1}}>
\widetilde \Omega^{a+1}_{Y'}
\end{CD}
$$ 
Here, $\phi_i$ and $\psi_i$ are $\mathcal O_{Y'}$-homomorphisms 
constructed in Proposition \ref{135} for any $i$. 
Since we assume that $l=p$, the above diagram is commutative. 
Therefore, we obtain the desired morphisms of complexes 
$\phi$ and $\psi$. 
\end{proof}

As an application of Proposition \ref{14}, we can prove the 
$E_1$-degeneration of Hodge to de Rham type spectral 
sequence for 
toric polyhedra. 

\begin{thm-a}[{cf.~Theorem \ref{01}}]\label{15} 
Let $Y=Y(\Phi)$ be a complete 
toric polyhedron. 
Then the spectral sequence 
$$
E^{a,b}_1=H^b(Y, \widetilde \Omega^a_Y)
\Rightarrow \mathbb H^{a+b} 
(Y, \widetilde \Omega ^{\bullet}_Y)
$$
degenerates at the $E_1$-term. 
\end{thm-a}
\begin{proof}
The following proof is well known. 
See, for example, 
the proof of Theorem 4 in \cite{btlm}. 
We assume that $l=p>0$, where 
$p$ is the characteristic of $k$. 
Then, by Proposition \ref{14}, 
\begin{eqnarray*}
\sum _{a+b=n}\dim _k E^{a,b}_{\infty} =
\dim _k \mathbb H^n(Y, \widetilde \Omega^{\bullet}_Y)
=\dim _k \mathbb H^n(Y, F_*\widetilde \Omega^{\bullet}_Y)\\
\geq \sum _{a+b=n}\dim _k H^b(Y', \widetilde \Omega^a_{Y'})
=\sum _{a+b=n} \dim _k E^{a,b}_1. 
\end{eqnarray*}
In general, $\sum _{a+b=n}\dim _k 
E^{a, b}_{\infty} \leq \sum _{a+b=n}\dim _kE^{a, b}_1$. 
Therefore, $E^{a, b}_{\infty} \simeq 
E^{a, b}_1$ holds and 
the spectral sequence degenerates at $E_1$. 
When the characteristic of $k$ is zero, 
we can assume that everything is defined over $\mathbb Q$. 
Moreover, we can construct a toric polyhedron $\mathcal Y$ 
defined over $\mathbb Z$ such that 
$Y=\mathcal Y\times _{\Spec \mathbb Z}\Spec \mathbb Q$. 
By applying the above $E_1$-degeneration 
on a general fiber of $f:{\mathcal Y}\to {\Spec\mathbb Z}$ and 
the base change theorem, 
we obtain that 
$\sum _{a+b=n}
\dim _{\mathbb Q}E^{a, b}_1=\dim _{\mathbb Q} \mathbb H^n(Y, \widetilde 
\Omega^{\bullet}_Y)$. 
In particular, 
$\sum _{a+b=n}
\dim _{k}E^{a, b}_1=\dim _k \mathbb H^n(Y, \widetilde 
\Omega^{\bullet}_Y)$ and 
we have the desired $E_1$-degeneration over $k$. 
\end{proof}

We close this section with the following two remarks on 
Ishida's results. 

\begin{rem-a} 
If $k=\mathbb C$ 
and $\Delta$ consists 
of a finite number of cones, then Ishida's de Rham complex 
$\widetilde \Omega^{\bullet}_Y$ is canonically isomorphic to 
the 
Du Bois complex 
$\underline{{\underline{\Omega}}}^{\bullet}_Y$ 
(see Theorem 4.1 
in \cite{ishida}). 
Therefore, the $E_1$-degeneration in Theorem \ref{15} 
was known when $k=\mathbb C$. 
We note that $\mathbb H^{a+b}(Y, \widetilde \Omega^{\bullet}_Y)$ 
is isomorphic to $H^{a+b} (Y, \mathbb C)$ in this case.  
\end{rem-a} 

\begin{rem-a}
Let $Y=Y(\Phi)$ be a toric polyhedron. 
In \cite[p.130]{ishida}, Ishida introduced 
a complex 
$C^{\bullet}(\Phi^{(2)}, \mathcal O\otimes \varLambda ^a)$. 
For the definition and the basic properties, 
see \cite[Sections 2 and 3]{ishida}. 
Note that $\widetilde \Omega^a_Y\to 
C^{0}(\Phi^{(2)}, \mathcal O\otimes \varLambda ^a)
\to \cdots \to 
C^{j}(\Phi^{(2)}, \mathcal O\otimes \varLambda ^a)
\to \cdots$ is a resolution of $\widetilde \Omega^a_Y$, 
that is, $\widetilde \Omega^a_Y\simeq 
\mathcal H^0(C^{\bullet}(\Phi^{(2)}, \mathcal O\otimes \varLambda ^a))$ 
and $\mathcal H^i(C^{\bullet}(\Phi^{(2)}, 
\mathcal O\otimes \varLambda ^a))=0$ 
for $i\ne 0$ (cf.~\cite[Proposition 2.4]
{ishida}). 
Assume that $Y$ is complete. 
Let $L$ be a nef line bundle on $Y$. 
Then it is not difficult to 
see that $\widetilde \Omega^a_Y
\otimes L\to C^{\bullet}(\Phi^{(2)}, \mathcal O\otimes \varLambda ^a)\otimes 
L$ is a $\Gamma$-acyclic resolution of $\widetilde \Omega^a_Y\otimes 
L$. 
Therefore, if $L$ is ample, then Theorem \ref{12} 
implies that 
$H^0(Y, \widetilde \Omega^a_Y\otimes L)=H^0
(D^{\bullet})$ 
and $H^i(Y, \widetilde \Omega^a_Y\otimes L)
=H^i(D^{\bullet})=0$ 
for $i\ne 0$, 
where $D^{\bullet}$ is a complex of $k$-vector spaces 
$\Gamma (Y, C^{\bullet}(\Phi^{(2)}, 
\mathcal O\otimes \varLambda ^a)\otimes L)$. 
\end{rem-a}

\section{Suppplements}\label{sec3}  

In this section, we make some remarks on my 
paper \cite{fujino1}. 
Let $X=X(\Delta)$ be a toric variety defined 
over a field $k$ of any characteristic. Note that 
$\Delta$ is not assumed to be finite in this section. 
First, we define 
$\widetilde \Omega^a_X(\log(A+B))(-A)$, which is a slight generalization 
of $\widetilde \Omega^a_X(\log B)$ in \cite[Definition 1.2]{fujino1}. 

\begin{defn}
Let $X$ be a toric variety and let $A$ and $B$ be reduced torus 
invariant Weil divisors on $X$ without common irreducible 
components. 
We put $W=X\setminus \Sing (X)$, where 
$\Sing (X)$ is the singular locus of $X$. 
Then we define 
$\widetilde \Omega^a_X(\log(A+B))(-A)=
\iota_*(\Omega^a_W(\log (A+B))\otimes \mathcal O_W(-A))$ for any 
$a\geq 0$, 
where $\iota:W\hookrightarrow X$ is the natural open immersion. 
\end{defn}

By the same argument as in \cite[Section 2]{fujino1} 
(see also Subsection \ref{sub22}), 
the split injection $\Omega^a_{T'}\to F_*\Omega^a_T$ induces the following 
split injection. 

\begin{prop}\label{a-1} 
Let $F:X\to X'$ be the $l$-times multiplication 
map. 
Then the split injection $\Omega ^a_{T'}\to F_*\Omega^a_T$ naturally 
induces the following split injection 
$\widetilde \Omega^a_{X'}(\log(A'+B'))(-A')\to 
F_*\widetilde \Omega^a_X(\log(A+B))(-A)$ for any $a\geq 0$. 
\end{prop}

The next proposition is obvious by the definition of 
the exterior derivative and 
the construction of the split injections 
in Proposition \ref{a-1} (cf.~Proposition \ref{14}). 

\begin{prop}\label{a-2}
We assume that 
$l=p>0$, where $p$ is the characteristic of $k$. 
Then there exist morphisms of complexes 
$$\phi:\bigoplus _{a\geq 0}\widetilde \Omega^a_{X'}(\log(A'+B'))(-A')[-a]\to 
F_*\widetilde \Omega^{\bullet}_X(\log(A+B))(-A)$$ 
and 
$$
\psi: 
F_*\widetilde \Omega^{\bullet}_X(\log(A+B))(-A)\to 
\bigoplus _{a\geq 0}\widetilde \Omega^a_{X'}(\log(A'+B'))(-A')[-a] 
$$ such that 
the composition $\psi\circ \phi$ is a quasi-isomorphism. 
We note that the complex 
$\bigoplus _{a\geq 0}\widetilde \Omega^a_{X'}(\log(A'+B'))(-A')[-a]$ 
has zero differentials. 
\end{prop}

The following $E_1$-degeneration is a direct consequence 
of Proposition \ref{a-2}. 
See the proof of Theorem \ref{15}. 

\begin{thm}[cf.~Theorem \ref{b-II}]\label{a-III} 
Let $X$ be a complete 
toric variety and let $A$ and $B$ be reduced torus 
invariant Weil divisors on $X$ without common irreducible 
components. 
Then the spectral sequence 
$$
E^{a,b}_1=H^b(X, \widetilde \Omega^a_{X}(\log (A+B))(-A))
\Rightarrow \mathbb H^{a+b}(X, \widetilde \Omega^{\bullet}
_X(\log (A+B))(-A)) 
$$ 
degenerates at the $E_1$-term. 
\end{thm}

\begin{rem}
If $k=\mathbb C$ and $X$ is non-singular and complete, 
then it is well known that 
$\mathbb H^{a+b}(X, \Omega^{\bullet}_X)=H^{a+b} 
(X, \mathbb C)$, 
$\mathbb H^{a+b}(X, \Omega^{\bullet}_X(\log B))=H^{a+b} 
(X\setminus B, \mathbb C)$, 
and 
$\mathbb H^{a+b}(X, \Omega^{\bullet}_X(\log A)\otimes 
\mathcal O_X(-A))=H^{a+b}_{c}  
(X\setminus A, \mathbb C)$, 
where $H^{a+b}_c(X\setminus A, \mathbb C)$ is the 
cohomology group with compact support. 
\end{rem}

Finally, we state a generalization 
of \cite[Theorem 1.1]{fujino1}. 
The proof is obvious. 
See also Theorem \ref{sugo} below. 

\begin{thm}[cf.~Theorem \ref{b-I}]\label{a-I}
Let $X$ be a toric variety and let $A$ and $B$ be reduced torus 
invariant Weil divisors on $X$ without common irreducible 
components. 
Let $L$ be a line bundle on $X$. If 
$H^i(X, \widetilde {\Omega}^a_{X}(\log(A+B))(-A)\otimes L^
{\otimes l})=0$ for 
some positive integer $l$, 
then $H^i(X, \widetilde {\Omega}^a_{X}(\log (A+B))(-A)\otimes L)=0$. 
\end{thm}

Some other vanishing theorems in \cite{fujino1} can 
be generalized by 
using Theorem \ref{a-I}. 
We leave the details for the reader's exercise. 

\section{Koll\'ar type vanishing theorems 
and extension theorem}\label{sec-kol}

In this section, we treat a variant of the method in \cite{fujino1}. 
Here, every toric variety is defined over a field $k$ of any characteristic 
and a fan is not necessarily finite. 
Let $f:Z\to X$ be a toric morphism of finite type. 
Then we have the following commutative 
diagram of $l$-times multiplication maps. 
$$
\begin{CD}
Z@>{F^{Z}}>> Z' \\ 
@V{f}VV @VV{f'}V \\ 
X@>>{F^{X}}> X'
\end{CD}
$$ 
This means that $F^X:X\to X'$ and $F^Z:Z\to Z'$ are the 
$l$-times multiplication maps 
explained in \ref{sub22} and 
that $F^X\circ f=f'\circ F^Z$. 
Let $\mathcal F$ be a coherent sheaf on $Z$ such 
that 
there exists a split injection 
$\alpha: \mathcal F'\to F^{Z}_*\mathcal F$. 
Then we have an obvious lemma. 

\begin{lem}\label{51} 
We have a split injection 
$$
\beta=R^jf'_*\alpha: R^jf'_*\mathcal F'\to 
F^{X}_*R^jf_*\mathcal F 
$$ 
for any $j$. 
\end{lem} 
\begin{proof}
Since 
$F^X$ and $F^Z$ are finite, we have the following isomorphisms 
$$
F^X_*R^jf_*\mathcal F\simeq R^j(F^X\circ f)_*\mathcal F
\simeq R^j(f'\circ F^Z)_*\mathcal F\simeq R^jf'_*(F^Z_*\mathcal F) 
$$ 
by Leray's spectral sequence. 
Therefore, we obtain a split injection 
$$
\beta=R^jf'_*\alpha: R^jf'_*\mathcal F'\to 
F^{X}_*R^jf_*\mathcal F 
$$ 
for any $j$. 
\end{proof}

Let $L$ be a line bundle on $X$. 
Then we obtain the following useful proposition. 

\begin{prop}\label{52} 
If $H^i(X, R^jf_*\mathcal F\otimes L^{\otimes l})=0$ for some 
positive integer $l$, 
then $H^i(X, R^jf_*\mathcal F\otimes L)=0$. 
\end{prop}
\begin{proof} 
Let $F^X$ be the $l$-times multiplication map. 
As usual, we have 
\begin{align*}
H^i(X, R^jf_*\mathcal F\otimes L)&\simeq 
H^i(X', R^jf'_*\mathcal F'\otimes L')
\\ 
&\subset 
H^i(X', F^X_*R^jf_*\mathcal F\otimes L')\\&\simeq 
H^i(X, R^jf_*\mathcal F\otimes L^{\otimes l}) 
\end{align*} 
because $(F^X)^*L'\simeq L^{\otimes l}$. 
So, we obtain the desired statement. 
\end{proof}

Therefore, we get a very powerful vanishing theorem. 

\begin{thm}\label{sugo}
Let $f:Z\to X$ be a proper toric morphism 
and let $A$ and $B$ be reduced 
torus invariant Weil divisors on $Z$ without common irreducible 
components. 
Assume that $\pi:X\to S$ is a projective 
toric morphism and $L$ is a $\pi$-ample line bundle on $X$. 
Then $R^i\pi_*(L\otimes R^jf_*\widetilde {\Omega}^a_Z(\log (A+B))(-A))=0$ 
for any $i>0$, $j\geq 0$, and $a\geq 0$. 
\end{thm}

\begin{proof}The problem is local. So, we can assume that 
$S$ is affine. 
We put $\mathcal F=\widetilde {\Omega}^a_Z(\log (A+B))(-A)$. Then, 
this is a direct consequence of Proposition \ref{a-1} and 
Proposition \ref{52} by Serre's vanishing theorem. 
\end{proof} 

We obtain Koll\'ar type vanishing theorem for toric varieties 
as a special case of Theorem \ref{sugo}. 

\begin{cor}[Koll\'ar type vanishing theorem]\label{c43}
Let $f:Z\to X$ be a proper toric morphism and 
let $B$ be a reduced torus invariant Weil divisor on $Z$. 
Assume that $X$ is projective and $L$ is 
an ample line bundle on $X$. 
Then $H^i(X, R^jf_*\mathcal O_Z(K_Z+B)\otimes L)=0$ for 
any $i>0$ and $j\geq 0$.  
\end{cor} 

\begin{proof} 
It is sufficient to put $a=\dim Z$ in Theorem \ref{sugo}. 
\end{proof}

The next theorem is one of the main results of this paper. 
See also Theorem \ref{431}. 

\begin{thm}[{cf.~\cite[Theorem 5.1]{mustata}}]\label{531} 
Let $\pi:X\to S$ be a proper toric morphism 
and $Y=Y(\Phi)$ a toric polyhedron on $X=X(\Delta)$. 
Let $L$ be a $\pi$-ample line bundle on $X$. 
Let $\mathcal I_{Y}$ be the defining ideal sheaf of $Y$ on $X$. 
Then $R^i\pi_*(\mathcal I_{Y}\otimes L)=0$ for 
any $i>0$. 
Since $R^i\pi_*L=0$ for 
any $i>0$, we have that 
$R^i(\pi|_Y)_*(L|_Y)=0$ for 
any $i>0$ and 
that the restriction map $\pi_*L\to (\pi|_Y)_*(L|_Y)$ 
is surjective. 
\end{thm}

\begin{proof} 
If $Y=X$, then there is nothing to prove. 
So, we can assume that $Y\subsetneq X$. 
Let $f:V\to X$ be a toric resolution 
such that $K_V+E=f^*(K_X+D)$ and 
that $\Supp (f^{-1}(Y))$ is a simple normal crossing 
divisor on $V$. 
We decompose $E=E_1+E_2$, where 
$E_1=\Supp (f^{-1}(Y))$ and $E_2=E-E_1$. 

\begin{claim}
We have an isomorphism 
$\mathcal I_Y\simeq f_*\mathcal O_V(-E_1)$. 
\end{claim}
\begin{proof}[Proof of Claim] 
By the definition of $E_1$, $f:V\to X$ induces 
a morphism $f:E_1\to Y$. 
We consider the following commutative diagram. 
$$
\xymatrix{
0\ar[r] &\mathcal I_Y\ar[d]\ar[r]&\mathcal O_X\ar[d]^{\simeq}\ar[r]&\mathcal O_Y\ar[d]
\ar[r] &0\\
0\ar[r] & f_*\mathcal O_V(-E_1)
\ar[r]&\mathcal O_X\ar[r]&f_*\mathcal O_{E_1}\ar[r] &\cdots
}
$$
Since $\mathcal O_Y\to f_*\mathcal O_{E_1}$ is injective, 
we have $\mathcal I_Y\simeq f_*\mathcal O_V(-E_1)$. 
\end{proof}
By the vanishing theorem (cf.~Theorem \ref{sugo} and 
Corollary \ref{c43}), we obtain 
that $$R^i\pi_*(f_*\mathcal O_V(-E_1)\otimes L)\simeq 
R^i\pi_*(\mathcal I_Y\otimes L)=0$$ for 
any $i>0$ because $-E_1\sim K_V+E_2$. 
The other statements are obvious by exact sequences.  
\end{proof}

\section{Toric polyhedra as quasi-log varieties}\label{sec4}

In this section, all (toric) varieties are assumed to 
be of finite type over the complex number 
field $\mathbb C$ to use the results in \cite{fujino2}. 
We will explain the background and motivation of the results 
obtained in the previous sections. 
Note that this section is independent of the other 
sections. 
We quickly review the notation of the log minimal model 
program. 

\begin{notation}
Let $V$ be a normal variety and let $B$ be 
an effective 
$\mathbb Q$-divisor on $V$ such that 
$K_V+B$ is $\mathbb Q$-Cartier. 
Then we can define 
the {\em{discrepancy}} $a(E, V, B)\in \mathbb Q$ for any prime 
divisor $E$ over $V$. 
If $a(E, V, B)\geq -1$ for any $E$, 
then $(V, B)$ is called {\em{log canonical}}. 
Let $(V, B)$ be a log canonical 
pair. 
If $E$ is a prime divisor over $V$ such that 
$a(E, V, B)=-1$, then 
$c_V(E)$ is called {\em{log canonical center}} 
of $(V, B)$, where 
$c_V(E)$ is the closure 
of the image of $E$ on $V$. 
\end{notation}

Let $X=X(\Delta)$ be a toric variety 
and let $D$ be the complement 
of the big torus. 
Then the next proposition 
is well known. 
So, we omit the proof. 

\begin{prop}\label{411} 
The pair $(X, D)$ is log canonical 
and $K_X+D\sim 0$. 
Let $W$ be a closed subvariety of $X$. 
Then, $W$ is a log canonical center of $(X, D)$ if 
and only if $W=V(\sigma)$ for some $\sigma \in \Delta\setminus \{0\}$. 
\end{prop}

By Proposition \ref{411} and 
adjunction in \cite[Theorem 4.4]{ambro} and 
\cite[Theorem 3.12]{fujino3}, 
we have the following useful theorem.  

\begin{thm}\label{422} 
Let $Y=Y(\Phi)$ be a toric polyhedron 
on $X$. 
Then, the log canonical pair $(X, D)$ induces a 
natural quasi-log structure on $(Y, 0)$. 
Note that $(Y, 0)$ has only 
qlc singularities. 
\end{thm}
Here, we do not explain the 
definition of quasi-log varieties. 
It is because it is very difficult to 
grasp. 
See the introduction of \cite{fujino3} and 
\ref{hoso} below. 
The essential point of the theory 
of quasi-log varieties is contained in the proof of 
Theorem \ref{431} below. 
The following theorem:~Theorem \ref{431} is 
my motivation for Theorem \ref{531}. 
It depends on the deep results 
obtained in \cite{fujino2}. 

\begin{thm}[{cf.~\cite[Theorem 5.1]{mustata}}]\label{431} 
Let $\pi:X\to S$ be a proper toric morphism and 
$Y=Y(\Phi)$ a toric polyhedron on $X=X(\Delta)$. 
Let $M$ be a Cartier divisor on $X$ such that 
$M$ is $\pi$-nef and $\pi$-big 
and $M|_{V(\sigma)}$ is $\pi$-big 
for any $\sigma \in \Delta \setminus \Phi$. 
Let $\mathcal I_{Y}$ be the defining ideal sheaf of $Y$ on $X$. 
Then $R^i\pi_*(\mathcal I_{Y}\otimes \mathcal O_X(M))=0$ for 
any $i>0$. 
Since $R^i\pi_*\mathcal O_X(M)=0$ for 
any $i>0$, we have that 
$R^i\pi_*\mathcal O_Y(M)=0$ for 
any $i>0$ and 
that the restriction map 
$\pi_*\mathcal O_X(M)\to \pi_*\mathcal O_Y(M)$ 
is surjective. 
\end{thm}

\begin{proof}[Sketch of the proof] 
If $Y=X$, then there is nothing to prove. 
So, we can assume that $Y\subsetneq X$.  
Let $f:V\to X$ be a toric resolution 
such that $K_V+E=f^*(K_X+D)$ and 
that $\Supp (f^{-1}(Y))$ is a simple normal crossing 
divisor on $V$. 
We decompose $E=E_1+E_2$, where 
$E_1=\Supp (f^{-1}(Y))$ 
and $E_2=E-E_1$. 
We consider the short exact sequence 
$$
0\to \mathcal O_V(-E_1)\to \mathcal O_V\to 
\mathcal O_{E_1}\to 0. 
$$ 
Then we obtain the exact sequence 
$$0\to f_*\mathcal O_V(-E_1)\to \mathcal O_X\to f_*\mathcal 
O_{E_1}\to R^1f_*\mathcal O_V(-E_1)\to \cdots. 
$$ 
Since $-E_1\sim K_V+E_2$, $R^1f_*\mathcal O_V(-E_1)\simeq 
R^1f_*\mathcal O_V(K_V+E_2)$ and every 
non-zero local section of $R^1f_*\mathcal O_V(-E_1)$ contains 
in its support the $f$-image 
of some strata of $(V, E_2)$ (see, for example, 
\cite[Theorem 7.4]{ambro} or \cite[Theorem 3.13]{fujino3}). 
Note that $W$ is a stratum of $(V, E_2)$ if and only if 
$W$ is $V$ or a log canonical 
center of $(V, E_2)$. 
On the other hand, the support of $f_*\mathcal O_{E_1}$ is contained 
in $Y$. Therefore, the connecting homomorphism 
$f_*\mathcal O_{E_1}\to R^1f_*\mathcal O_V(-E_1)$ is a $0$-map. 
Thus, we obtain $$0\to f_*\mathcal O_V(-E_1)\to 
\mathcal O_X\to \mathcal O_Y\to 0$$ and 
$\mathcal I_Y\simeq f_*\mathcal O_V(-E_1)$. 
We consider $f^*M\sim f^*M-E_1-(K_V+E_2)$. 
By the vanishing theorem (see \cite{fujino2} and 
\cite[Theorem 3.13]{fujino3}), we 
obtain $R^i\pi_*(f_*\mathcal O_V(f^*M-E_1))\simeq 
R^i\pi_*(\mathcal I_Y\otimes \mathcal O_X(M))=0$ for any 
$i>0$. 
The other statements are obvious by exact sequences.  
\end{proof}

\begin{rem} 
In Theorem \ref{431}, 
by the Lefschetz principle, we can replace the 
base field $\mathbb C$ with a field $k$ of characteristic zero. 
I believe that 
Theorem \ref{431} 
holds true for toric varieties defined over 
a field $k$ of any characteristic. 
However, I did not check it. 
\end{rem}

\begin{rem}
In the proof of Theorem \ref{431}, 
we did not use the fact that $\pi:X\to S$ is {\em{toric}}. 
We just needed the properties in Proposition \ref{411}. 
\end{rem}

\begin{say}[Comments on Theorem \ref{422}]\label{hoso} 
We freely use the notation in the proof of Theorem \ref{431}. 
We assume that $Y\subsetneq X$. 
Then we have the following properties. 
\begin{enumerate}
\item $g^*0\sim K_{E_1}+{E_2}|_{E_1}$, 
where $g=f|_{E_1}:E_1\to Y$. 
\item ${E_2}|_{E_1}$ is reduced and $g_*\mathcal O_{E_1}\simeq 
\mathcal O_Y$. 
\item The collection of subvarieties $\{V(\sigma)\}_{\sigma \in \Phi}$ 
coincides with the image of torus invariant irreducible 
subvarieties of $V$ which are contained in $E_1$. 
\end{enumerate}
Therefore, $Y$ is a {\em{quasi-log variety}} 
with the {\em{quasi-log canonical class}} $0$ and the subvarieties 
$V(\sigma)$ for $\sigma \in \Phi$ are the {\em{qlc centers}} 
of $Y$. 
We sometimes call $g:(E_1, {E_2}|_{E_1})\to Y$ a 
{\em{quasi-log resolution}}. 
For the details, see \cite{fujino3}. 
\end{say}

\ifx\undefined\bysame
\newcommand{\bysame|{leavemode\hbox to3em{\hrulefill}\,}
\fi

\end{document}